\documentclass[12pt]{article}
\usepackage{amsmath}
\usepackage{amsthm}
\usepackage{amsfonts}
\usepackage{amssymb}
\usepackage{esint}
\usepackage{eucal}
\usepackage{mathrsfs}
\usepackage{graphicx,graphics}
\numberwithin{equation}{section}
\usepackage{graphicx,graphics}
\usepackage{pdfsync}
\usepackage{multirow}
\usepackage{endnotes}
\def \dis {\displaystyle}

\def \confai {-\kern -.5em\rightharpoonup}
\def \cqfd {\hfill$\Box$}
\def \div{\mbox{\rm div}}

\def \al {\alpha}
\def \be {\beta}
\def \ga {\gamma}

\def \de {\delta}
\def \De {\Delta}
\def \ep {\varepsilon}
\def \om {\omega}
\def \Om {\Omega}
\def \la {\lambda}

\def \ph {\varphi}
\def \th {\theta}

\def \si {\sigma}

\def \NN {\mathbb N}
\def \ZZ {\mathbb Z}

\def \RR {\mathbb R}
\def \CC {\mathbb C}

\def \M {\mathscr{M}}

\def \beq {\begin{equation}}
\def \eeq {\end{equation}}
\def \ba {\begin{array}}
\def \ea {\end{array}}
\def \bs {\bigskip}
\def \ms {\medskip}

\def \ecart {\noalign{\medskip}}
\newtheorem{Thm}{Theorem}[section]

\newtheorem{Pro}[Thm]{Proposition}
\newtheorem{Lem}[Thm]{Lemma}

\newtheorem{Adef}[Thm]{Definition}

\newtheorem{Arem}[Thm]{Remark}
\newenvironment{Rem}{\begin{Arem}\rm}{\end{Arem}}
\newtheorem{Aexa}[Thm]{Example}

\newtheorem{Anot}[Thm]{Notation}

\def \refe #1.{(\ref{#1})}
\def \reff #1.{figure~\ref{#1}}
\def \refs #1.{Section~\ref{#1}}
\def \refss #1.{Subsection~\ref{#1}}
\def \refD #1.{Definition~\ref{#1}}
\def \refT #1.{Theorem~\ref{#1}}
\def \refL #1.{Lemma~\ref{#1}}
\def \refC #1.{Corollary~\ref{#1}}
\def \refP #1.{Proposition~\ref{#1}}
\def \refPt #1.{Properties~\ref{#1}}
\def \refR #1.{Remark~\ref{#1}}
\def \refE #1.{Example~\ref{#1}}
\def \refN #1.{Notation~\ref{#1}}
\newcounter{marnote}

\title{First Bloch eigenvalue in high contrast media}
\author{Marc Briane\footnote{Institut de Recherche Math\'ematique de Rennes, INSA de Rennes, FRANCE -- mbriane@insa-rennes.fr} \and
Muthusamy Vanninathan\footnote{TIFR-CAM, Bangalore, INDIA -- vanni@math.tifrbng.res.in}}
\begin{document}
\maketitle
\begin{abstract}
This paper deals with the asymptotic behavior of the first Bloch eigenvalue in a heterogeneous medium with a high contrast $\ep Y$-periodic conductivity. When the conductivity is bounded in $L^1$ and the constant of the Poincar\'e-Wirtinger weighted by the conductivity is very small with respect to $\ep^{-2}$, the first Bloch eigenvalue converges as $\ep\to 0$ to a limit which preserves the second-order expansion with respect to the Bloch parameter. In dimension two the expansion of the limit can be improved until the fourth-order under the same hypotheses. On the contrary, in dimension three a fibers reinforced medium combined with a $L^1$-unbounded conductivity leads us to a discontinuity of the limit first Bloch eigenvalue as the Bloch parameter tends to zero but remains not orthogonal to the direction of the fibers. Therefore, the high contrast conductivity of the microstructure induces an anomalous effect, since for a given low-contrast conductivity the first Bloch eigenvalue is known to be analytic with respect to the Bloch parameter around zero.
\end{abstract}
\par\bs\noindent
\section{Introduction}
The oscillating operators of type
\beq\label{ade}
\nabla\cdot\big(a(x/\de)\nabla\cdot\big),\quad\mbox{as }\de\to 0,
\eeq
for coercive and bounded $Y$-periodic matrix-valued functions $a(y)$ in $\RR^d$, which model the conduction in highly heterogeneous media, have been widely studied since the seminal work \cite{BLP} based on an asymptotic expansion of \refe{ade}.. In the end of the nineties an alternative approach was proposed in \cite{CoVa} using the Bloch wave decomposition. More precisely, this method consists in considering the discrete spectrum $\big(\la_m(\eta),\phi_m(\eta)\big)$, $m\geq 1$, of the translated complex operator (see \cite{COV1} for the justification)
\beq\label{Aeta}
A(\eta):=-\,(\nabla+i\eta)\cdot\big[a(y)\,(\nabla+i\eta)\big],\quad\mbox{for a given }\eta\in\RR^d.
\eeq
It was proved in \cite{CoVa} that the first Bloch pair $\big(\la_1(\eta),\phi_1(\eta)\big)$ actually contains the essential informations on the asymptotic analysis of the operator \refe{ade}., and are analytic with respect to the Bloch parameter $\eta$ in a neighborhood of $0$. Moreover, by virtue of \cite{CoVa,COV1} it turns out that the first Bloch eigenvalue satisfies the following expansion in terms of the so-called Burnett coefficients:
\beq\label{expla1eta}
\la_1(\eta)=q\eta\cdot\eta+D(\eta\otimes\eta):(\eta\otimes\eta)+o(|\eta|^4),
\eeq
where $q$ is the homogenized positive definite conductivity associated with the oscillating sequence $a(x/\de)$, and $D$ is the fourth-order dispersion tensor which has the remarkable property to be non-positive for any conductivity matrix $a$ (see \cite{COV2}).
\par
The expansion \refe{expla1eta}. has been investigated more deeply in one dimension~\cite{CSSV2} and in low contrast regime \cite{CSSV1}. It is then natural to study the behavior of \refe{expla1eta}. in high contrast regime. This is also motivated by the fact that the homogenization of operators \refe{ade}. with high contrast coefficients may induce nonlocal effects in dimension three as shown in \cite{FeKh,BeBo,CESe,BrTc},
while the two-dimensional case of \cite{Bri3,BrCa1,BrCa2} is radically different. We are interested in knowing the consequences of these effects in Bloch waves analysis.
\bs\par
The aim of the paper is then to study the asymptotic behavior of the first Bloch eigenvalue in the presence of high contrast conductivity coefficients. In particular we want to specify the validity of expansion \refe{expla1eta}. in high contrast regime. To this end we consider an $\ep Y$-periodic matrix conductivity $a^\ep$ which is equi-coercive but not equi-bounded with respect to $\ep$, namely $\|a^\ep\|_{L^\infty}\to\infty$ as $\ep\to 0$. The classical picture is an $\ep Y$-periodic two-phase microstructure, one of the phase has a conductivity which blows up as $\ep$ tends to $0$. More precisely, we will study the limit behavior of the first Bloch eigenvalue $\la_1^\ep(\eta)$ associated with $a^\ep$, and its expansion
\beq\label{expla1epeta}
\la_1^\ep(\eta)=q^\ep\eta\cdot\eta+D^\ep(\eta\otimes\eta):(\eta\otimes\eta)+o(|\eta|^4).
\eeq
\par
In Section \ref{s.L1bd}, we prove that in any dimension (see Theorem~\ref{thm1}), if the conductivity $a^\ep$ is bounded in $L^1$ and the constant of the Poincar\'e-Wirtinger inequality weighted by $a^\ep$ is an $o(\ep^{-2})$ (see \cite{Bri1} for an example), then the first Bloch eigenvalue $\la_1^\ep(\eta)$ associated with $a^\ep$ converges to some limit $\la_1^*(\eta)$ which satisfies
\beq
\la_1^*(\eta)=q^*\eta\cdot\eta,\quad\mbox{for small enough }|\eta|,
\eeq
where $q^*$ is the limit of the homogenized matrix $q^\ep$ in \refe{expla1epeta}.. Moreover, in dimension two and under the same assumptions we show that the tensor $D^\ep$ tends to $0$, which thus implies that the fourth-order expansion \refe{expla1epeta}. of $\la_1^\ep(\eta)$ converges to the fourth-order expansion of its limit. We can also refine the two-dimensional case by relaxing the $L^1$-boundedness of $a^\ep$ by the sole convergence of $q^\ep$ (see Theorem~\ref{thm12d}). 
\par
In Section~\ref{s.L1unbd}, we show that the previous convergences do not hold generally in dimension three when $a^\ep$ is not bounded in $L^1$. We give a counter-example (see Theorem~\ref{thm2}) which is based on the fibers reinforced structure introduced first in \cite{FeKh} to derive nonlocal effects in high contrast homogenization. This is the main result of the paper. We show the existence of a jump at $\eta=0$ in the limit $\la_1^*(\eta)$ of the first Bloch eigenvalue. Indeed, when the radius of the fibers has a critical size and $\eta$ is not orthogonal to their direction, the first Bloch eigenvector $\psi_1^\ep$ is shown to converge weakly in $H^1_{\rm loc}(\RR^3;\CC)$ to some function $\psi_1^*$ solution of
\beq\label{equpsi*}
-\,\De\psi_1^*+\ga\,\psi_1^*=\la_1^*(\eta)\,\psi_1^*\quad\mbox{in }\RR^3,
\eeq
where
\beq
\ga=\lim_{\eta\to 0}\la_1^*(\eta)\neq\la_1^*(0)=0.
\eeq
Therefore, contrary to the analyticity of $\eta\mapsto\la_1^\ep(\eta)$ which holds for fixed $\ep$, the limit $\la_1^*$ of the first Bloch eigenvalue is not even continuous at $\eta=0$! On the other hand, the zero-order term in limit \refe{equpsi*}. is linked to the limit zero-order term obtained in \cite{BeBo,BrTc} under the same regime, for the conduction equation with the conductivity $a^\ep$ but with a Dirichlet boundary condition. Here, the periodicity condition satisfied by the function $y\mapsto e^{-i\,\eta\cdot y}\,\psi_1^\ep(y)$ (in connection with the translated operator \refe{Aeta}.) is quite different and more delicate to handle. Using an estimate of the Poincar\'e-Wirtinger inequality weighted by $a^\ep$ and the condition  that $\eta$ is not orthogonal to the direction of the fibers, we can get the limit in the Radon measures sense of the eigenvector $\psi_1^\ep$ rescaled in the fibers.
\subsection{Notations}
\begin{itemize}
\item $\ep$ denotes a small positive number such that $\ep^{-1}$ is an integer.
\item $\left(e_1,\dots,e_d\right)$ denotes the canonical basis of $\RR^d$.
\item $\cdot$ denotes the canonical scalar product in $\RR^d$.
\item $:$ denotes the canonical scalar product in $\RR^{d\times d}$.
\item $Y$ denotes the cube $(0,2\pi)^d$ in $\RR^d$.
\item $H^1_\sharp(Y)$ denotes the space of the $Y$-periodic functions which belong to $H^1_{\rm loc}(\RR^d)$.
Similarly, $L^p_\sharp(Y)$, for $p\geq 1$, denotes the space of the $Y$-periodic functions which belong to $L^p_{\rm loc}(\RR^d)$, and $C^k_\sharp(Y)$, for $k\in\NN$, denotes the space of the $C^k$-regular $Y$-periodic functions in $\RR^d$.
\item For any $\eta\in\RR^d$, $H^1_\eta(Y;\CC)$ denotes the space of the functions $\psi$ such that
\beq\label{H1eta}
\big(x\mapsto e^{-i\,x\cdot\eta}\,\psi(x)\big)\in H^1_\sharp(Y;\CC).
\eeq
Similarly, $L^p_\eta(Y;\CC)$, for $p\geq 1$, denotes the set denotes the set associated with the space $L^p_\sharp(Y;\CC)$, and $C^k_\eta(Y;\CC)$, for $k\in\NN$, the set associated with the space $C^k_\sharp(Y;\CC)$.
\item For any open set $\Om$ of $\RR^d$, $BV(\Om)$ denotes the space of the functions in $L^2(\Om)$ the gradient of which is a Radon measure on $\Om$. 
\end{itemize}
\section{The case of $L^1$-bounded coefficients}\label{s.L1bd}
Let $\ep>0$ be such that $\ep^{-1}$ is an integer.
Let $A^\ep$ be a $Y$-periodic measurable real matrix-valued function satisfying
\beq\label{Aep}
\left(A^\ep\right)^T(y)=A^\ep(y)\quad\mbox{and}\quad\al\,I_d\leq A^\ep(y)\leq\be_\ep\,I_d\qquad\mbox{a.e. }y\in\RR^d,
\eeq
where $\al$ is a fixed positive number and $\be_\ep$ is a sequence in $(0,\infty)$ which tends to $\infty$ as $\ep\to 0$.
Let $a^\ep$ be the rescaled matrix-valued function defined by
\beq\label{aep}
a^\ep(x):=A^\ep\left({x\over\ep}\right)\quad\mbox{for }x\in\RR^d.
\eeq
Define the effective conductivity $q^\ep$ by
\beq\label{qep}
q^\ep\la:=\fint_Y A^\ep\big(\la+\nabla X_\la^\ep\big)\,dy\quad\mbox{for }\la\in\RR^d,
\eeq
where $X_\la^\ep$ is the unique solution in $H^1_\sharp(Y)/\RR$ of the equation
\beq\label{Xepla}
\div\left(A^\ep\la+A^\ep\nabla X_\la^\ep\right)=0\quad\mbox{in }\RR^d.
\eeq
For a fixed $\ep>0$, the constant matrix $q^\ep$ is the homogenized matrix associated with the oscillating sequence $A^\ep({x\over\de})$ as $\de\to 0$, according to the classical homogenization periodic formula (see, {\it e.g.}, \cite{BLP}).
\par
Consider for $\eta\in\RR^d$, the first Bloch eigenvalue $\la_1^\ep(\eta)$ associated with the conductivity $a^\ep$ by
\beq\label{laepeta}
\la_1^\ep(\eta):=\min\left\{\int_Y a^\ep\nabla\psi\cdot \nabla\overline{\psi}\,dx\;:\;\psi\in H^1_\eta(Y;\CC)\mbox{ and }\int_Y|\psi|^2\,dx=1\right\}.
\eeq
A minimizer $\psi^\ep$ of \refe{laepeta}. solves the variational problem
\beq\label{vppsiep}
\int_Y a^\ep\nabla\psi^\ep\cdot \nabla\overline{\psi}\,dx=\la^\ep_1(\eta)\int_Y\psi^\ep\,\overline{\psi}\,dx,\quad\forall\,\psi\in H^1_\eta(Y;\CC),
\eeq
with
\beq\label{psiep}
\psi^\ep\in H^1_\eta(Y;\CC)\quad\mbox{and}\quad\int_Y|\psi^\ep|^2\,dx=1.
\eeq
An alternative definition for $\psi^\ep$ is given by the following result:
\begin{Pro}\label{pro.psiep}
The variational problem \refe{vppsiep}. is equivalent to the equation in the distributional sense
\beq\label{eqpsiep}
-\,\div\left(a^\ep\nabla\psi^\ep\right)=\la^\ep_1(\eta)\,\psi^\ep\quad\mbox{in }\RR^d.
\eeq
\end{Pro}
\begin{proof}
Let $\psi^\ep\in H^1_\eta(Y)$ be a solution of \refe{vppsiep}. and let $\ph$ be a function in $C^\infty_c(\RR^d)$. Writing $\psi^\ep=e^{i\,x\cdot\eta}\,\ph^\ep$ with $\ph^\ep\in H^1_\sharp(Y;\CC)$, and putting the function $\psi\in C^\infty_\eta(Y;\CC)$ defined by
\[
\psi(x):=\sum_{k\in\ZZ^d}e^{-i\,2\pi k\cdot\eta}\,\overline{\ph}(x+2\pi k)
=e^{i\,x\cdot\eta}\sum_{k\in\ZZ^d}e^{-i\,(x+2\pi k)\cdot\eta}\,\overline{\ph}(x+2\pi k),
\]
as test function in \refe{vppsiep}., we have by the $Y$-periodicity of $a^\ep$ (recall that $\ep$ is an integer),
\[
\ba{l}
\dis \int_Y a^\ep\nabla\psi^\ep\cdot\nabla\overline{\psi}\,dx
\\ \ecart
\dis =\sum_{k\in\ZZ^d}\int_Y a^\ep\left(\nabla\ph^\ep+i\,\eta\,\ph^\ep\right)\cdot
\left[\nabla\big(e^{i(x+2\pi k)\cdot\eta}\,\ph(x+2\pi k)\big)-i\,\eta\,\big(e^{i(x+2\pi k)\cdot\eta}\,\ph(x+2\pi k)\big)\right]dx
\\ \ecart
\dis =\dis \sum_{k\in\ZZ^d}\int_{2\pi k+Y} a^\ep\left(\nabla\ph^\ep+i\,\eta\,\ph^\ep\right)\cdot
\left[\nabla\big(e^{i\,x\cdot\eta}\,\ph\big)-i\,\eta\,\big(e^{i\,x\cdot\eta}\,\ph\big)\right]dx
\\ \ecart
\dis =\int_{\RR^d}a^\ep\nabla\psi^\ep\cdot\nabla\ph\,dx,
\ea
\]
and
\[
\ba{ll}
\dis \int_Y \psi^\ep\,\overline{\psi}\,dx
& \dis =\sum_{k\in\ZZ^d}\int_Y \ph^\ep\,e^{i(x+2\pi k)}\,\ph(x+2\pi k)\,dx
\\ \ecart
& =\dis \sum_{k\in\ZZ^d}\int_{2\pi k+Y} \ph^\ep\,e^{i\,x\cdot\eta}\,\ph\,dx=\int_{\RR^d}\psi^\ep\,\ph\,dx.
\ea
\]
Hence, we get that
\[
\int_{\RR^d}a^\ep\nabla\psi^\ep\cdot\nabla\ph\,dx=\la^\ep_1(\eta)\int_{\RR^d}\psi^\ep\,\ph\,dx,\quad\mbox{for any }\ph\in C^\infty_c(\RR^d),
\]
which yields equation \refe{eqpsiep}..
\par
Conversely, assume that $\psi^\ep$ is a solution of \refe{eqpsiep}.. Consider $\psi\in C^\infty_\eta(Y;\CC)$, and for any integer $n\geq 1$, a function $\th_n\in C^\infty_c(\RR^d)$ such that
\[
\th_n=1\;\;\mbox{in }[-2\pi n,2\pi n]^d,\quad \th_n=0\;\;\mbox{in }\RR^d\setminus[-2\pi(n+1),2\pi(n+1)]^d,
\quad|\nabla\th_n|\leq 1\;\;\mbox{in }\RR^d.
\]
Putting $\ph:=\th_n\,\overline{\psi}$ as test function in \refe{eqpsiep}., we have as $n\to\infty$ and by the $Y$-periodicity of
$a^\ep\nabla\psi^\ep\cdot\nabla\overline{\psi}$,
\[
\ba{ll}
\dis {1\over (2n)^d}\int_{\RR^d}a^\ep\nabla\psi^\ep\cdot\nabla(\th_n\,\overline{\psi})\,dx
& \dis ={1\over (2n)^d}\sum_{k\in\{-n,\dots,n-1\}^d}\int_{2\pi k+Y}a^\ep\nabla\psi^\ep\cdot\nabla\overline{\psi}\,dx+o_n(1)
\\ \ecart
& \dis =\int_Y a^\ep\nabla\psi^\ep\cdot\nabla\overline{\psi}\,dx +o_n(1),
\ea
\]
and by the $Y$-periodicity of $\psi^\ep\,\overline{\psi}$,
\[
\ba{ll}
\dis {1\over (2n)^d}\int_{\RR^d}\psi^\ep\,\th_n\,\overline{\psi}\,dx
& \dis ={1\over (2n)^d}\sum_{k\in\{-n,\dots,n-1\}^d}\int_{2\pi k+Y}\psi^\ep\,\overline{\psi}\,dx+o_n(1)
\\ \ecart
& \dis =\int_Y \psi^\ep\,\overline{\psi}\,dx +o_n(1).
\ea
\]
Therefore, it follows that $\psi^\ep$ is solution of the variational problem \refe{vppsiep}..
\end{proof}
Note that for a fixed $\ep>0$, the oscillating sequence $a^\ep({x\over\de})=A^\ep({x\over\ep\de})$ has the same homogenized limit as $A^\ep({x\over\de})$ when $\de$ tends to $0$, namely the constant matrix $q^\ep$ defined by \refe{qep}.. Hence, the asymptotic expansion in $\eta$ of the first Bloch eigenvalue derived in \cite{COV2} reads as
\beq\label{asylaepeta}
\la_1^\ep(\eta)=q^\ep\eta\cdot\eta+D^\ep(\eta\otimes\eta):(\eta\otimes\eta)+O(|\eta|^6),
\eeq
where $D^\ep$ is a non-positive fourth-order tensor defined in formula \refe{Dep}. below.
\par
When $a^\ep$ is not too high, we have the following asymptotic behavior for $\la_1^\ep(\eta)$:
\begin{Thm}\label{thm1}
Assume that the sequence $a^\ep$ of \refe{aep}. is bounded in $L^1(Y)$.
\begin{itemize}
\item If $d=2$, then there exists a subsequence of $\ep$, still denoted by $\ep$, such that the sequence $q^\ep$ converges to some $q^*$ in $\RR^{2\times 2}$. Moreover, we have for any $\eta\in\RR^2$,
\beq\label{limlaepeta}
\lim_{\ep\to 0}\la_1^\ep(\eta)=\min\left\{\int_Y q^*\nabla\psi\cdot\nabla\overline{\psi}\,dx\;:\;\psi\in H^1_\eta(Y;\CC)\mbox{ and }\int_Y|\psi|^2\,dx=1\right\},
\eeq
and for small enough $|\eta|$,
\beq\label{limlaepeta0}
\lim_{\ep\to 0}\la_1^\ep(\eta)=q^*\eta\cdot\eta.
\eeq
\item If $d\geq 2$, under the extra assumption that that for any $\la\in\RR^d$,
\beq\label{Cepla}
C^\ep_\la:=\max\left\{\int_Y(A^\ep\la\cdot\la)\,V^2\,dy\;:\;V\in H^1_\sharp(Y),\ \int_Y A^\ep\nabla V\cdot\nabla V\,dy=1\right\}\ll{1\over\ep^2},
\eeq
and \refe{limlaepeta}., \refe{limlaepeta0}. still hold. Moreover, if $d=2$ we have
\beq\label{limDep}
\lim_{\ep\to 0}\big(D^\ep(\eta\otimes\eta):(\eta\otimes\eta)\big)=0,\quad\forall\,\eta\in\RR^d.
\eeq
\end{itemize}
\end{Thm}
Using a more sophisticated approach we can relax in dimension two the $L^1(Y)$-boundedness of $a^\ep$:
\begin{Thm}\label{thm12d}
Assume that $d=2$ and that the sequence $q^\ep$ converges to $q^*$ in $\RR^{2\times 2}$. Then, the limits \refe{limlaepeta}. and \refe{limlaepeta0}. still hold.
\end{Thm}
\begin{Rem}
The constant $C^\ep_\la$ of \refe{Cepla}. is the best constant of the Poincar\'e-Wirtinger inequality weighted by $A^\ep$. The condition $\ep^2\,C^\ep_\la\to 0$ was first used in \cite{Bri1} to prevent the appearance of nonlocal effects in the homogenization of the conductivity equation with $a^\ep$. Under this assumption the first Bloch eigenvalue and its second-order expansion converge as $\ep$ tends to $0$ in any dimension $d\geq 2$. The case $d=2$ is quite particular since it is proved in \cite{BrCa2} that nonlocal effects cannot appear. This explains {\it a posteriori} that the first Bloch eigenvalue has a good limit behavior under the $L^1(Y)$-boundedness of $a^\ep$ (Theorem~\ref{thm1}), or the sole boundedness of $q^\ep$ (Theorem~\ref{thm12d}). Note that the second condition is more general than the first one due to the estimate \refe{qepmin}. below.
\end{Rem}
\noindent
{\bf Proof of Theorem~\ref{thm1}.}
\par\ms\noindent
{\it The case $d=2$}: The proof is divided in two parts. In the first part we determine the limit of the eigenvalue problem \refe{vppsiep}.. The second part provides the limit of the minimization problem~\refe{laepeta}..
\par
The matrix $q^\ep$ of \refe{qep}. is also given by the minimization problem for any $\la\in\RR^d$:
\beq\label{qepmin}
q^\ep\la\cdot\la=\min\left\{\fint_Y A^\ep(\la+\nabla V)\cdot(\la+\nabla V)\,dy\;:\;V\in H^1_\sharp(Y)\right\}\leq\fint_Y A^\ep\la\cdot\la\,dy
\eeq
which is bounded. Therefore, up to a subsequence $q^\ep$ converges to some $q^*$ in $\RR^{d\times d}$.
\par
To obtain the limit behavior of \refe{vppsiep}. we need to consider the rescaled test functions $w^\ep_j$, $j=1,2$, associated with the cell problem \refe{Xepla}. and defined by
\beq\label{wepj}
w^\ep_j(x):=x_j+\ep\,X^\ep_{e_j}\left({x\over\ep}\right)\quad\mbox{for }x\in\RR^d.
\eeq
Since by the $\ep Y$-periodicity of $\nabla w^\ep_j$, $j=1,2$, and by \refe{qepmin}.
\beq\label{estwepj}
\fint_Y a^\ep\nabla w^\ep_j\cdot\nabla w^\ep_j=q^\ep_{jj}\leq c,
\eeq
the sequence $w^\ep_j$ is bounded in $H^1_{\rm loc}(\RR^2)$ and thus converges weakly to $x_i$ in $H^1_{\rm loc}(\RR^2)$.
By the Corollary~2.3 of \cite{BrCa2} (which is specific to dimension two), the sequence $w^\ep:=(w^\ep_1,w^\ep_2)$ converges uniformly to the identity function locally in $\RR^2$. Moreover, since $\ep^{-1}$ is an integer and the functions $X^\ep_{e_j}$ are $Y$-periodic, we have for any $x\in\RR^2$ and $k\in\ZZ^2$,
\[
w^\ep_j(x+2\pi k)=x_j+2\pi k_j+\ep\,X^\ep_{e_j}\left({x+2\pi k\over\ep}\right)=x_j+2\pi k_j+\ep\,X^\ep_{e_j}\left({x\over\ep}\right)=w^\ep_j(x)+2\pi k_j,
\]
or equivalently,
\beq\label{wepk}
w^\ep(x+2\pi k)=w^\ep(x)+2\pi k,\quad\forall\,(x,k)\in\RR^2\times\ZZ.
\eeq
This implies that for any $\chi\in C^1_\eta(Y;\CC)$, the function $\chi(w^\ep)$ belongs to $H^1_\eta(Y;\CC)$ (see \refe{H1eta}.).
\par
On the other hand, the eigenvalue $\la^\ep_1(\eta)$ \refe{laepeta}. is bounded due to the $L^1(Y)$-boundedness of $a^\ep$, and thus converges up to a subsequence to some number $\la^*_1(\eta)\geq 0$. Hence, the sequence $\psi^\ep$ is bounded in $H^1_\eta(Y;\CC)$, and thus converges weakly up to a subsequence to some function $\psi^*$ in $H^1_\eta(Y;\CC)$.
Then, putting $\chi(w^\ep)$ as test function in \refe{vppsiep}., using the uniform convergence of $w^\ep$ and the convergence of $\psi^\ep$ to $\psi^*$, we get that
\beq\label{est21}
\int_Y a^\ep\nabla\psi^\ep\cdot\nabla w^\ep_j\,\overline{\partial_j\chi(w^\ep)}\,dx
=\int_Y a^\ep\nabla\psi^\ep\cdot\nabla w^\ep_j\,\partial_j\overline{\chi}\,dx+o(1)
=\la^*_1(\eta)\int_Y\psi^*\,\overline{\chi}\,dx+o(1).
\eeq
Next, let us apply the div-curl approach of \cite{Bri3,BrCa1}. To this end, since by \refe{Xepla}. and \refe{wepj}. the current $a^\ep\nabla w^\ep_j$ is divergence free, we may consider a stream function $\tilde{w}^\ep_j$ associated with $a^\ep\nabla w^\ep_j$ such that
\beq\label{twepj}
a^\ep\nabla w^\ep_j=\nabla^\perp\tilde{w}^\ep_j:=\begin{pmatrix} -\,\partial_2\tilde{w}^\ep_j \\ \partial_1\tilde{w}^\ep_j\end{pmatrix}
\quad\mbox{a.e. in }\RR^2.
\eeq
By the Cauchy-Schwarz inequality combined with \refe{estwepj}. and the $L^1(Y)$-boundedness of $a^\ep$, the function $\tilde{w}^\ep_j$ is bounded in $BV_{\rm loc}(\RR^2)$. Moreover, due to the periodicity the sequence $\nabla\tilde{w}^\ep_j$ induces no concentrated mass in the space $\M(\RR^2)^2$ of the Radon measures on $\RR^2$. Therefore, by the Lions concentration-compactness lemma \cite{PLL} $\tilde{w}^\ep_j$ converges strongly in $L^2_{\rm loc}(\RR^2)$ to some function $\tilde{w}^j$ in~$BV_{\rm loc}(\RR^2)$. By the $\ep Y$-periodicity of $a^\ep\nabla w^\ep_j$ and the definition \refe{qep}. of $q^\ep$, we also have in the weak-$*$ sense of the Radon measures
\beq\label{Dtwj}
a^\ep\nabla w^\ep_j\;\rightharpoonup\;\nabla^\perp\tilde{w}_j=\lim_{\ep\to 0}\left(\fint_Y A^\ep\big(e_j+\nabla X^\ep_{e_j}\big)\,dy\right)=q^*e_j
\quad\mbox{weakly in }\M(\RR^2)^2*.
\eeq
On the other hand, integrating by parts using that $a^\ep\nabla w^\ep_j$ is divergence free and $\psi^\ep\,\partial_j\overline{\chi}$ is $Y$-periodic, then applying the strong convergence of $\tilde{w}^\ep_j$ in $L^2(Y)$ and \refe{Dtwj}., it follows that (with the summation over repeated indices)
\[
\ba{ll}
\dis \int_Y a^\ep\nabla w^\ep_j\cdot\nabla\psi^\ep\,\partial_j\overline{\chi}\,dx
& =-\dis \int_Y a^\ep\nabla w^\ep_j\cdot\nabla(\partial_j\overline{\chi})\,\psi^\ep\,dx
=\int_Y \tilde{w}^\ep_j\,\nabla^T\psi^\ep\cdot\nabla(\partial_j\overline{\chi})\,dx
\\ \ecart
& \dis =\int_Y \tilde{w}_j\,\nabla^T\psi^*\cdot\nabla(\partial_j\overline{\chi})\,dx+o(1)
=-\int_Y q^*e_j\cdot\nabla(\partial_j\overline{\chi})\,\psi^*\,dx+o(1)
\\ \ecart
& \dis =\int_Y q^*\nabla\psi^*\cdot\nabla\overline{\chi}\,dx+o(1).
\ea
\]
This combined with \refe{est21}. and a density argument yields the limit variational problem
\beq\label{vppsi*}
\int_Y q^*\nabla\psi^*\cdot\nabla\overline{\chi}\,dx=\la^*_1(\eta)\int_Y\psi^*\,\overline{\chi}\,dx,\quad\forall\,\chi\in H^1_\eta(Y;\CC),
\eeq
where by Rellich's theorem and \refe{psiep}. the limit $\psi^*$ of $\psi^\ep$ satisfies
\beq\label{psi*}
\psi^*\in H^1_\eta(Y;\CC)\quad\mbox{and}\quad\int_Y|\psi^*|^2\,dx=1.
\eeq
\par
It remains to prove that
\beq\label{la*eta}
\la_1^*(\eta)=\min\left\{\int_Y q^*\nabla\psi\cdot\nabla\overline{\psi}\,dx\;:\;\psi\in H^1_\eta(Y;\CC)\mbox{ and }\int_Y|\psi|^2\,dx=1\right\}.
\eeq
To this end consider a covering of $Y$ by $n\geq 1$ two by two disjoint cubes $Q^n_k$ of same size, and $n$ smooth functions $\th^n_k$, for $1\leq k\leq n$, such that
\beq\label{Qnk}
\th^n_k\in C^1_0\big(Q^n_k;[0,1]\big)\quad\mbox{and}\quad\sum_{k=1}^n \th^n_k\;\mathop{\longrightarrow}_{n\to\infty}\;1\;\;\mbox{strongly in }L^2(Y).
\eeq
For $\chi\in C^1_\eta(Y;\CC)$ with a unit $L^2(Y)$-norm, consider the approximation $\chi^\ep_n$ of $\chi$ defined by
\beq\label{chiep}
\chi^\ep_n(x):=\nu^\ep_n\left(\chi(x)+\ep\,e^{i\,x\cdot\eta}\sum_{k=1}^n\th^n_k(x)\,X^\ep_{\xi^n_k}\left({x\over\ep}\right)\right),\quad\mbox{where}\quad
\xi^n_k:=\fint_{Q^n_k}e^{-i\,x\cdot\eta}\,\nabla\chi(x)\,dx,
\eeq
and $\nu^\ep_n>0$ is chosen in such a way that $\chi^\ep_n$ has a unit $L^2(Y)$-norm.
Since $\ep^{-1}$ is an integer, the function $\chi^\ep_n$ belongs to $H^1_\eta(Y;\CC)$ and can thus be used as a test function in problem \refe{laepeta}.. Then, by \refe{Qnk}. we have
\[
\ba{l}
\dis \la^\ep_1(\eta)\leq\int_Y a^\ep\nabla\chi^\ep_n\cdot\nabla\overline{\chi^\ep_n}\,dx
\\ \ecart
\dis \leq(\nu^\ep_n)^2\int_Y a^\ep\left(\nabla\chi+e^{i\,x\cdot\eta}\sum_{k=1}^n\th^n_k\,\nabla X^\ep_{\xi^n_k}\left({x\over\ep}\right)\right)\cdot\overline{\left(\nabla\chi+e^{i\,x\cdot\eta}\sum_{k=1}^n\th^n_k\,\nabla X^\ep_{\xi^n_k}\left({x\over\ep}\right)\right)}dx+o(1)
\\ \ecart
\dis =(\nu^\ep_n)^2\int_Y a^\ep\left(R^n+e^{i\,x\cdot\eta}\sum_{k=1}^n\th^n_k\,\big(\xi^n_k+\nabla X^\ep_{\xi^n_k}\big)\left({x\over\ep}\right)\right)\cdot\overline{\left(R^n+e^{i\,x\cdot\eta}\sum_{k=1}^n\th^n_k\,\big(\xi^n_k+\nabla X^\ep_{\xi^n_k}\big)\left({x\over\ep}\right)\right)}dx
\\
\hfill +\,o(1),
\ea
\]
\beq\label{Rnk}
\mbox{where}\quad R^n:=\nabla\chi-e^{i\,x\cdot\eta}\sum_{k=1}^n\th^n_k\,\xi^n_k\in C^0(\RR^2)^2.
\eeq
Passing to the limit as $\ep\to 0$ in the previous inequality, and using \refe{Qnk}., the $L^1(Y)$-boundedness combined with the $Y$-periodicity of $A^\ep$, $A^\ep\nabla X^\ep_\la$, $A^\ep\nabla X^\ep_\la\cdot\nabla X^\ep_\la$, and the convergence
\beq\label{limDwxink}
\ba{l}
\dis a^\ep\,\big(\xi^n_k+\nabla X^\ep_{\xi^n_k}\big)\left({x\over\ep}\right)\cdot\big(\xi^n_k+\nabla X^\ep_{\xi^n_k}\big)\left({x\over\ep}\right)
=\left(A^\ep\big(\xi^n_k+\nabla X^\ep_{\xi^n_k}\big)\cdot\big(\xi^n_k+\nabla X^\ep_{\xi^n_k}\big)\right)\left({x\over\ep}\right)
\\ \ecart
\dis \;\rightharpoonup\;\lim_{\ep\to 0}\left(\fint_Y A^\ep\big(\xi^n_k+\nabla X^\ep_{\xi^n_k}\big)\cdot\big(\xi^n_k+\nabla X^\ep_{\xi^n_k}\big)\,dy\right)=q^*\xi^n_k\cdot\overline{\xi^n_k}\quad\mbox{weakly in }\M(\bar{Y})\,*,
\ea
\eeq
it follows that
\beq\label{est22}
\la^*_1(\eta)\leq\int_Yq^*\left(e^{i\,x\cdot\eta}\sum_{k=1}^n\th^n_k\,\xi^n_k\right)\cdot\overline{\left(e^{i\,x\cdot\eta}\sum_{k=1}^n\th^n_k\,\xi^n_k\right)}
+c\int_Y\left(|R^n|^2+|R^n|\right)dx.
\eeq
Therefore, since the sequence $R^n$ of \refe{Rnk}. converges strongly to $0$ in $L^2(Y;\CC)^2$, passing to the limit as $n\to\infty$ in \refe{est22}. we get that for any $\chi\in C^1_\eta(Y;\CC)$ with a unit $L^2(Y)$-norm,
\beq\label{estla*eta}
\la^*_1(\eta)\leq\int_Y q^*\nabla\chi\cdot\nabla\overline{\chi}\,dx.
\eeq
Using a density argument the inequality \refe{estla*eta}. combined with the limit problem \refe{vppsi*}. implies the desired formula \refe{la*eta}.. Moreover, due to the uniqueness of \refe{la*eta}. in term of $q^*$ the limit \refe{limlaepeta}. holds for any $\eta\in\RR^2$, and for the whole sequence $\ep$ such that $q^\ep$ converges to $q^*$. Finally, decomposing formula \refe{la*eta}. in Fourier's series and using Parseval's identity we obtain that equality \refe{limlaepeta0}. holds for any $\eta\in\RR^2$ with small enough norm.
\par\ms\noindent
{\it The case $d\geq 2$ under assumption \refe{Cepla}.}: First, note that the proof of the inequality \refe{estla*eta}. in the previous case actually holds for any dimension. Therefore, it is enough to obtain the limit eigenvalue problem \refe{vppsi*}. to conclude to the minimization formula \refe{la*eta}.. To this end, applying the homogenization Theorem~2.1 of \cite{Bri1} to the linear equation \refe{eqpsiep}., we get the limit equation
\beq\label{eqpsi*}
\int_{\RR^2} q^*\nabla\psi^*\cdot\nabla\ph\,dx=\la^*_1(\eta)\int_{\RR^2}\psi^*\,\ph\,dx,\quad\forall\,\ph\in C^\infty_c(\RR^2),
\eeq
which is equivalent to \refe{la*eta}. by Proposition~\ref{pro.psiep}.
\par
It thus remains to prove~\refe{limDep}. when $d=2$, which is also a consequence of \refe{Cepla}..
By \cite{COV2} we have
\beq\label{Dep}
D^\ep(\eta\otimes\eta):(\eta\otimes\eta)
=-\fint_Y a^\ep\,\nabla\left(\chi^\ep_{2,\eta}-{1\over 2}\,(\chi^\ep_{1,\eta})^2\right)\cdot\nabla\left(\chi^\ep_{2,\eta}-{1\over 2}\,(\chi^\ep_{1,\eta})^2\right)dx,
\eeq
where, taking into account \refe{Xepla}. and \refe{wepj}.,
\beq\label{chiep1}
\chi^\ep_{1,\eta}(x):=\ep\,X^\ep_\eta\left({x\over\ep}\right)=\eta_1\left(w^\ep_1-x_1\right)+\eta_2\left(w^\ep_2-x_2\right)
\quad\mbox{for }x\in\RR^2,
\eeq
and $\chi^\ep_{2,\eta}$ is the unique function in $H^1_\sharp(Y)$ with zero $Y$-average, solution of
\beq\label{chiep2}
-\,\div\left(a^\ep\nabla\chi^\ep_{2,\eta}\right)
=a^\ep\eta\cdot\eta-q^\ep\eta\cdot\eta+a^\ep\eta\cdot\nabla\chi^\ep_{1,\eta}+\div\left(\chi^\ep_{1,\eta}\,a^\ep\eta\right)
\quad\mbox{in }\RR^d.
\eeq
Consider the partition of $Y$ by the small cubes $2\pi\ep k+\ep Y$, for $k\in\{0,\dots,\ep^{-1}-1\}^2$, and define from $\chi^\ep_{j,\eta}$, $j=1,2$, the associated average function
\beq\label{bchiepj}
\breve{\chi}^\ep_{j,\eta}:=\sum_{k\in\{0,\dots,\ep^{-1}-1\}^2}\left(\fint_{2\pi\ep k+\ep Y}\chi^\ep_{j,\eta}\,dx\right)1_{2\pi \ep k+\ep Y},
\eeq
where $1_E$ denotes the characteristic function of the set $E$. Then, $\ep$-rescaling estimate \refe{Cepla}. we get that
\beq\label{estchiepj}
\int_Y a^\ep\eta\cdot\eta\left(\chi^\ep_{j,\eta}-\breve{\chi}^\ep_{j,\eta}\right)^2dx
\leq\ep^2\,C^\ep_\eta\int_Y a^\ep\nabla\chi^\ep_{j,\eta}\cdot\nabla\chi^\ep_{j,\eta}\,dx.
\eeq
Also note that $\breve{\chi}^\ep_{1,\eta}=0$, and since $\chi^\ep_{2,\eta}$ has zero $Y$-average, we have
\beq\label{mbchiepj}
\int_Y \breve{\chi}^\ep_{2,\eta}(x)\,Z\left({x\over\ep}\right)dx=0,\quad\forall\,Z\in L^2_\sharp(Y).
\eeq
Putting $\chi^\ep_{2,\eta}$ as test function in equation \refe{chiep2}., then using equality \refe{mbchiepj}. and the Cauchy-Schwarz inequality combined with the $L^1(Y)$-boundedness of $a^\ep$ and estimates \refe{estchiepj}., \refe{estwepj}., we obtain that
\[
\ba{l}
\dis \int_Y a^\ep\nabla\chi^\ep_{2,\eta}\cdot\nabla\chi^\ep_{2,\eta}\,dx
\\ \ecart
\dis =\int_Y a^\ep\eta\cdot\eta\left(\chi^\ep_{2,\eta}-\breve{\chi}^\ep_{2,\eta}\right)dx
+\int_Y a^\ep\eta\cdot\nabla\chi^\ep_{1,\eta}\left(\chi^\ep_{2,\eta}-\breve{\chi}^\ep_{2,\eta}\right)dx
-\int_Y a^\ep\eta\cdot\nabla\chi^\ep_{2,\eta}\left(\chi^\ep_{1,\eta}-\breve{\chi}^\ep_{1,\eta}\right)dx
\\ \ecart
\dis \leq c\left(\ep^2\,C^\ep_\eta\int_Y a^\ep\nabla\chi^\ep_{2,\eta}\cdot\nabla\chi^\ep_{2,\eta}\,dx\right)^{1\over 2}
\left[1+\left(\int_Y a^\ep\nabla\chi^\ep_{1,\eta}\cdot\nabla\chi^\ep_{1,\eta}\,dx\right)^{1\over 2}\right]
\\ \ecart
\dis \leq c\,\ep^2\,C^\ep_\eta\left(\int_Y a^\ep\nabla\chi^\ep_{2,\eta}\cdot\nabla\chi^\ep_{2,\eta}\,dx\right)^{1\over 2}.
\ea
\]
This together with assumption \refe{Cepla}. yields
\beq\label{estchiep2}
\lim_{\ep\to 0}\left(\int_Y a^\ep\nabla\chi^\ep_{2,\eta}\cdot\nabla\chi^\ep_{2,\eta}\,dx\right)=0.
\eeq
On the other hand, by \refe{chiep1}. and the Corollary~2.3 of \cite{BrCa2} (see the previous step) the sequence $\chi^\ep_{1,\eta}$ converges uniformly to $0$ in $Y$. At this level the dimension two is crucial. This combined with the Cauchy-Schwarz inequality and the $L^1(Y)$-boundedness of $a^\ep\nabla\chi^\ep_{j,\eta}\cdot\nabla\chi^\ep_{j,\eta}$ implies that
\beq\label{estchiep1}
\lim_{\ep\to 0}\left(\int_Y a^\ep\nabla\chi^\ep_{1,\eta}\cdot\nabla\chi^\ep_{j,\eta}\,(\chi^\ep_{1,\eta})^k\,dx\right)=0
\quad\mbox{for }j,k\in\{1,2\}.
\eeq
Therefore, passing to the limit in \refe{Dep}. thanks to \refe{estchiep2}. and \refe{estchiep1}. we get the desired convergence \refe{limDep}., which concludes the proof of Theorem~\ref{thm1}. \cqfd
\par\bs
To prove Theorem~\ref{thm12d} we need the following result the main ingredients of which are an estimate due to Manfredi \cite{Man} and a uniform convergence result of \cite{BrCa3}:
\begin{Lem}\label{lemM}
Let $\Om$ be a domain of $\RR^2$, and let $\si^\ep$ be a sequence of symmetric matrix-valued functions in $\RR^{2\times 2}$ such that $\al\,I_2\leq \si^\ep(x)\leq\be_\ep\,I_2$ a.e. $x\in\Om$, for a constant $\al>0$ independent of $\ep$ and a constant $\be_\ep>\al$. Let $f^\ep$ be a strongly convergent sequence in $W^{-1,p}(\Om)$ for some $p>2$. Consider a bounded sequence $u^\ep$ in $H^1(\Om)$ solution of the equation $-\,\div\left(\si^\ep\nabla u^\ep\right)=f^\ep$ in $\Om$. Then, up to a subsequence $u^\ep$ converges uniformly in any compact set of $\Om$.
\end{Lem}
\begin{proof}
On the one hand, let $D$ be a disk of $\Om$ such that $\bar{D}\subset\Om$, and let $u^\ep_D$ be the solution in $H^1_0(D)$ of the equation $-\,\div\left(\si^\ep\nabla u^\ep_D\right)=f^\ep$ in $D$. Since $u^\ep_D\equiv 0$ converges uniformly on $\partial D$ and $f^\ep$ converges strongly in $W^{-1,p}(\Om)$, by virtue of the Theorem~2.7 of \cite{BrCa3}, up to a subsequence $u^\ep_D$ converges weakly in $H^1(D)$ and uniformly in $\bar{D}$.
\par
On the other hand, the function $v^\ep:=u^\ep-u^\ep_D$ is bounded in $H^1(D)$ and solves the equation $\div\left(\si^\ep\nabla v^\ep\right)=0$ in $D$. By the De Giorgi-Stampacchia regularity theorem for second-order elliptic equations, $v^\ep$ is H\"older continuous in $D$ and satisfies the maximum principle in any disk of $D$. Hence, the function $v^\ep$ is continuous and weakly monotone in $D$ in the sense of \cite{Man}. Therefore, the estimate (2.5) of \cite{Man} implies that for any $x_0\in D$, there exists a constant $r>0$ such that
\beq\label{estM}
\forall\,x,y\in D(x_0,r),\quad\big|v^\ep(x)-v^\ep(y)\big|\leq{C\,\|\nabla v^\ep\|_{L^2(D)^2}\over\big[\ln\left(4r/|x-y|\right)\big]^{1\over 2}}
\leq{C\,\|\nabla v^\ep\|_{L^2(D)^2}\over\left(\ln 2\right)^{1\over 2}},
\eeq
where $D(x_0,r)$ is the disk centered on $x_0$ of radius $r$, and $C>0$ is a constant depending only on dimension two.
The sequence $v^\ep-v^\ep(x_0)$ is bounded in $D(x_0,r)$, independently of $\ep$ by the right-hand term of \refe{estM}.. This combined with the boundedness of $v^\ep$ in $L^2(D)$ implies that $v^\ep$ is bounded uniformly in $D(x_0,r)$. Moreover, estimate \refe{estM}. shows that the sequence $v^\ep$ is equi-continuous in $D(x_0,r)$. Then, by virtue of Ascoli's theorem together with a diagonal extraction procedure, up to a subsequence $v^\ep$ converges uniformly in any compact set of $D$. So does the sequence $u^\ep=u^\ep_D+v^\ep$. Again using a diagonal procedure from a countable covering of $\Om$ by disks $D$, there exists a subsequence of $\ep$, still denoted by $\ep$, such that $u^\ep$ converges uniformly in any compact set of $\Om$.
\end{proof}
\noindent
{\bf Proof of Theorem~\ref{thm12d}.} We have only to show that the limit $\psi^*$ of the eigenvector $\psi^\ep$ satisfying \refe{vppsiep}. and \refe{psiep}. is solution of \refe{vppsi*}.. Indeed, the proof of inequality \refe{estla*eta}. follows from the convergence of $q^\ep$ thanks to limit \refe{limDwxink}.. as shown in the proof of Theorem~\ref{thm1}.
\par
First of all, for $\chi\in C^1_\eta$ with a unit $L^2(Y)$-norm, $\chi(w^\ep)$ converges uniformly tends to $\chi$ in $Y$ due to the uniform convergence of $w^\ep$ (see the proof of Theorem~\ref{thm1} or apply Lemma~\ref{lemM}). Then, using successively the minimum formula \refe{laepeta}. with the test function $\psi:=\chi(w^\ep)$, the Cauchy-Schwarz inequality, the $\ep Y$-periodicity of $a^\ep\nabla w^\ep_j\cdot\nabla w^\ep_j$ and the boundedness of $q^\ep$, we have  (with the summation over repeated indices)
\[
\ba{ll}
\la^\ep_1(\eta) & \dis \leq{1\over\|\chi(w^\ep)\|^2_{L^2(Y)}}
\int_Y a^\ep\nabla w^\ep_j\cdot\nabla w^\ep_k\,\partial_j\chi(w^\ep)\,\overline{\partial_k\chi(w^\ep)}\,dx
\\ \ecart
& \dis \leq c\int_Y a^\ep\nabla w^\ep_j\cdot\nabla w^\ep_j\,dx\leq c\,{\rm tr}\left(q^\ep\right)\leq c.
\ea
\]
Hence, up to a subsequence $\la^\ep_1(\eta)$ converges to some $\la^*_1(\eta)$ in $\RR$. This combined with \refe{vppsiep}. and \refe{psiep}. implies that the eigenvector $\psi^\ep$ converges weakly to some $\psi^*$ in $H^1_{\rm loc}(\RR^2)$. Moreover, $\Re(\psi^\ep)$, $\Im(\psi^\ep)$ are solutions of equation \refe{eqpsiep}. with respective right-hand sides $\la^\ep_1(\eta)\,\Re(\psi^\ep)$, $\la^\ep_1(\eta)\,\Im(\psi^\ep)$ which are bounded in $H^1_{\rm loc}(\RR^2)$ thus in $W^{-1,p}_{\rm loc}(\RR^2)$ for any $p>2$. Therefore, thanks to Lemma~\ref{lemM} and up to extract a new subsequence, $\psi^\ep$ converges uniformly to $\psi^*$ in any compact set of $\RR^2$.
\par
On the other hand, for $\ph\in C^\infty_c(\RR^2)$, putting $\ph(w^\ep)$ as test function in equation \refe{eqpsiep}., using that $a^\ep\nabla w^\ep_j$ is divergence free (due to \refe{Xepla}. and \refe{wepj}.), and integrating by parts, we have (with the summation over repeated indices)
\beq\label{D2phweppsiep}
\ba{ll}
\dis \int_Y a^\ep\nabla\psi^\ep\cdot\nabla w^\ep_j\,\partial_j\ph(w^\ep)\,dx
& \dis =-\int_{\RR^2} a^\ep\nabla w^\ep_j\cdot\nabla w^\ep_k\,\partial^2_{jk}\ph(w^\ep)\,\psi^\ep\,dx
\\ \ecart
& \dis =\la^\ep_1(\eta)\int_{\RR^2}\psi^\ep\,\ph(w^\ep)\,dx.
\ea
\eeq
Then, passing to the limit in \refe{D2phweppsiep}. using the uniform convergences of $w^\ep, \psi^\ep$ combined with the convergences
\[
a^\ep\nabla w^\ep_j\cdot\nabla w^\ep_k\;\rightharpoonup\;
\lim_{\ep\to 0}\left(\fint_Y A^\ep\left(e_j+\nabla X^\ep_{e_j}\right)\cdot\left(e_k+\nabla X^\ep_{e_k}\right)\right)
=q^*_{jk}\quad\mbox{weakly in }\M(\RR^2)\,*,
\]
we get that
\beq\label{D2phpsi*}
-\int_{\RR^2} q^*_{jk}\,\partial^2_{jk}\ph\,\psi^*\,dx
=\la^*_1(\eta)\int_{\RR^2}\psi^*\,\ph\,dx.
\eeq
Finally, integrating by parts the left-hand side of \refe{D2phpsi*}. we obtain the limit equation \refe{eqpsi*}.,
which is equivalent to the limit eigenvalue problem \refe{vppsi*}. by virtue of Proposition~\ref{pro.psiep}. \cqfd
\section{Anomalous effect with $L^1$-unbounded coefficients}\label{s.L1unbd}
In this section we assume that $d=3$. Let $\ep>0$ be such that $\ep^{-1}$ is an integer.
Consider the fiber reinforced structure introduced in \cite{FeKh} and extended in several subsequent works \cite{BeBo, CESe, Bri2} to derive nonlocal effects in homogenization. Here we will consider this structure with a very high isotropic conductivity $a^\ep$ which is not bounded in $L^1(Y)$. More precisely, let $\om^\ep\subset Y$ be the $\ep Y$-periodic lattice composed by $\ep^{-2}$ cylinders of axes
\[
\big(2\pi k_1\ep+\pi\ep,2\pi k_2\ep+\pi\ep,0\big)+\RR\,e_3,\quad\mbox{for }(k_1,k_2)\in\{0,\dots,\ep^{-1}-1\}^2,
\]
of length $2\pi$, and of radius $\ep\,r_\ep$ such that
\beq\label{rep}
\lim_{\ep\to 0}\left({1\over2\pi\,\ep^2|\ln r_\ep|}\right)=\ga\in(0,\infty).
\eeq
The conductivity $a^\ep$ is defined by
\beq\label{aepbep}
a^\ep(x):=\left\{\ba{ll} \be_\ep & \mbox{if }x\in\om^\ep \\ \ecart 1 & \mbox{if }x\in Y\setminus\om^\ep \ea\right.
\quad\mbox{with}\quad\lim_{\ep\to 0}\be_\ep\,r_\ep^2=\infty,
\eeq
so that $a^\ep$ is not bounded in $L^1(Y)$.
\par
Then, we have the following result:
\begin{Thm}\label{thm2}
Assume that condition \refe{rep}. holds. Then, the first Bloch eigenvalue $\la^\ep_1(\eta)$ defined by \refe{laepeta}. with the conductivity $a^\ep$ of \refe{aepbep}. satisfies for any $\eta\in\RR^3$ with $\eta_3\notin\ZZ$,
\beq\label{laepetaga}
\lim_{\ep\to 0}\la_1^\ep(\eta)=\ga+\min\left\{\int_Y |\nabla\psi|^2\,dx\;:\;\psi\in H^1_\eta(Y;\CC)\mbox{ and }\int_Y|\psi|^2\,dx=1\right\},
\eeq
and for $|\eta|\leq{1\over 2}$ with $\eta_3\neq 0$,
\beq\label{laepetaga0}
\lim_{\ep\to 0}\la^\ep_1(\eta)=\ga+|\eta|^2.
\eeq
\end{Thm}
\begin{Rem}
For a fixed $\ep>0$, the function $\eta\mapsto\la^\ep_1(\eta)$ is analytic in a neighborhood of $0$. However, the limit $\la^*_1$ of $\la^\ep_1$ is not even continuous at $\eta=0$, since by \refe{asylaepeta}. and \refe{laepetaga0}. we have
\[
\la^*_1(0)=0\quad\mbox{while}\quad\lim_{\eta\to 0,\,\eta_3\neq 0}\la^*_1(\eta)=\ga>0.
\]
Contrary to the case of the $L^1$-bounded coefficients the $L^1(Y)$-unboundedness of $a^\ep$ induces a gap of the first Bloch eigenvalue. Therefore, the very high conductivity of the fiber structure deeply modifies the wave propagation in any direction $\eta$ such that $\eta_3\notin\ZZ$.
\end{Rem}
\noindent
{\bf Proof of Theorem~\ref{thm2}.} First we will determine the limit of the eigenvalue problem \refe{vppsiep}..
Following \cite{BeBo,BrTc} consider the function $\hat{v}^\ep$ related to the fibers capacity and defined by
\beq\label{hvep}
\hat{v}^\ep(x):=\hat{V}^\ep\left({x\over\ep}\right)\quad\mbox{for }x\in\RR^3,
\eeq
where $\hat{V}^\ep$ is the $y_3$-independent $Y$-periodic function defined in the cell period $Y$ by
\beq\label{hVep}
\hat{V}^\ep(y):=
\left\{\ba{cl}
0 & \mbox{if }\in[0,r_\ep]
\\ \ecart
\dis {\ln r-\ln r_\ep\over\ln R-\ln r_\ep} & \mbox{if }r\in(r_\ep,R)
\\ \ecart
1 & \mbox{if }r\geq R
\ea\right.
\quad\mbox{where}\quad r:=\sqrt{(y_1-\pi)^2+(y_2-\pi)^2},
\eeq
and $R$ is a fixed number in $(r_\ep,\pi)$. By a simple adaptation of the Lemma~2 of \cite{BrTc} combined with~\refe{rep}. the sequence $\hat{v}^\ep$ satisfies the following properties
\beq\label{conhvep}
\hat{v}^\ep=0\;\;\mbox{in }\om^\ep\quad\mbox{and}\quad\hat{v}^\ep\rightharpoonup 1\;\;\mbox{weakly in }H^1(Y),
\eeq
and for any bounded sequence $v^\ep$ in $H^1(Y)$, with ${1_{\om^\ep}\over|\om^\ep|}\,v^\ep$ bounded in $L^1(Y)$,
\beq\label{limhvepvep}
\nabla v^\ep\cdot\nabla\hat{v}^\ep-\ga\left(v^\ep-|Y|\,{1_{\om^\ep}\over|\om^\ep|}\,v^\ep\right)
\rightharpoonup 0\quad\mbox{weakly in }\M(\bar{Y})*.
\eeq
The last convergence involves the potential $v^\ep$ in the whole domain and the rescaled potential ${1_{\om^\ep}\over|\om^\ep|}\,v^\ep$ in the fibers set. In \cite{BeBo,BrTc,Bri2} it is proved that under assumption \refe{rep}. the homogenization of the conduction problem, with a Dirichlet boundary condition on the bottom of a cylinder parallel to the fibers, yields two different limit potentials inducing:
\begin{itemize}
\item either a nonlocal term if $a^\ep$ is bounded in $L^1$,
\item or only a zero-order term if $a^\ep$ is not bounded in $L^1$.
\end{itemize}
Here the situation is more intricate since the Dirichlet boundary condition is replaced by condition \refe{H1eta}.. This needs an alternative approach to obtain the boundedness of the potential $\psi^\ep$ solution of \refe{vppsiep}. and its rescaled version ${1_{\om^\ep}\over|\om^\ep|}\,\psi^\ep$.
\par
On the one hand, putting $e^{i\,x\cdot\eta}\,\hat{v}^\ep/\|e^{i\,x\cdot\eta}\,\hat{v}^\ep\|_{L^1(Y)}$ which is zero in $\om^\ep$, as test function in the minimization problem \refe{laepeta}. and using \refe{rep}. we get that $\la^\ep_1(\eta)$ is bounded. Hence, the sequence $\psi^\ep$ is bounded in $H^1_\eta(Y;\CC)$, and up to a subsequence converges weakly to some $\psi^*$ in $H^1_\eta(Y;\CC)$. On the other hand, the boundedness of ${1_{\om^\ep}\over|\om^\ep|}\,\psi^\ep$ in $L^1(Y;\CC)$ is more delicate to derive. To this end, we need the following Poincar\'e-Wirtinger inequality weighted by the conductivity $A^\ep(y):=a^\ep(\ep y)$:
\beq\label{PW}
\int_Y A^\ep\left|\,V-\fint_Y V\,dy\,\right|^2 dy\leq C\,|\ln r_\ep|\,\|A^\ep\|_{L^1(Y)}\int_Y A^\ep\,|\nabla V|^2\,dy,
\quad\forall\,V\in H^1(Y;\CC),
\eeq
which is an easy extension of the Proposition~2.4 in \cite{Bri1} to the case where $A^\ep$ is not bounded in~$L^1(Y)$.
Rescaling \refe{PW}. and using \refe{rep}. combined with the boundedness of $\la^\ep_1(\eta)$ we get that
\beq\label{estbpsiep}
\ba{ll}
\dis \int_Y a^\ep\,\big|\psi^\ep-\breve{\psi^\ep}\big|^2\,dx
& \dis \leq c\,\ep^2\,|\ln r_\ep|\,\|A^\ep\|_{L^1(Y)}\int_Y a^\ep\,|\nabla\psi^\ep|^2\,dx
\\ \ecart
&\dis \leq c\,\|A^\ep\|_{L^1(Y)}\int_Y a^\ep\,|\nabla\psi^\ep|^2\,dx\leq c\,\|a^\ep\|_{L^1(Y)},
\ea
\eeq
where for any $\chi\in L^2(Y;\CC)$, $\breve{\chi}$ denotes the piecewise constant function
\beq\label{bchiep}
\breve{\chi}:=\sum_{k\in\{0,\dots,\ep^{-1}-1\}^3}\left(\fint_{2\pi\ep k+\ep Y}\chi\,dx\right)1_{2\pi \ep k+\ep Y}.
\eeq
Then, from the Jensen inequality, the estimates $\be_\ep\,|\om^\ep|\sim\|a^\ep\|_{L^1(Y)}$ and \refe{estbpsiep}. we deduce that
\beq\label{estbpsiep2}
\fint_{\om^\ep}\big|\psi^\ep-\breve{\psi^\ep}\big|\,dx\leq\left(\fint_{\om^\ep}\big|\psi^\ep-\breve{\psi^\ep}\big|^2dx\right)^{1\over 2}
\leq c\left(\int_{\om^\ep}{a^\ep\over\|a^\ep\|_{L^1(Y)}}\,\big|\psi^\ep-\breve{\psi^\ep}\big|^2dx\right)^{1\over 2}\leq c.
\eeq
Moreover, since $\big|\om^\ep\cap(2\pi\ep k+\ep Y)\big|=\ep^3\,|\om^\ep|$ for any $k\in\{0,\dots,\ep^{-1}-1\}^3$, we have
\beq\label{estbpsiep3}
\ba{ll}
\dis \fint_{\om^\ep}\big|\breve{\psi^\ep}\big|\,dx &
\dis \leq\sum_{k\in\{0,\dots,\ep^{-1}-1\}^3}{1\over|\om^\ep|}
\int_{\om^\ep\cap(2\pi\ep k+\ep Y)}\left(\fint_{2\pi\ep k+\ep Y}\left|\psi^\ep\right|\right)dx
\\ \ecart
& \dis =\sum_{k\in\{0,\dots,\ep^{-1}-1\}^3}{1\over|Y|}\int_{2\pi\ep k+\ep Y}\left|\psi^\ep\right|
=\fint_Y\left|\psi^\ep\right|dx\leq c.
\ea
\eeq
Estimates \refe{estbpsiep2}. and \refe{estbpsiep3}. imply that the rescaled potential ${1_{\om^\ep}\over|\om^\ep|}\,\psi^\ep$ is bounded in $L^1(Y;\CC)$. Therefore, up to extract a new subsequence there exists a Radon measure $\tilde{\psi}^*$ on $\bar{Y}$ such that
\beq\label{tpsi*}
\tilde{\psi}^\ep:={1_{\om^\ep}\over|\om^\ep|}\,\psi^\ep\;\rightharpoonup\;\tilde{\psi}^*\quad\mbox{weakly in }\M(\bar{Y})\,*,
\eeq
or equivalently,
\beq\label{tph*}
\tilde{\ph}^\ep:=e^{-i\,x\cdot\eta}\, \tilde{\psi}^\ep\;\rightharpoonup\;\tilde{\ph}^*:=e^{-i\,x\cdot\eta}\,\tilde{\psi}^*\quad\mbox{weakly in }\M(\bar{Y})*.
\eeq
\par
Now, we have to evaluate the Radon measure $\tilde{\psi}^*$.
Let $\chi\in C^1_\eta(Y;\CC)$, since $1_{\om^\ep}$ is independent of the variable $x_3$ and the function $\psi^\ep\,\overline{\chi}$ is $Y$-periodic, an integration by parts yields
\beq\label{D3psiep}
\fint_{\om^\ep}\psi^\ep\,\partial_3\overline{\chi}\,dx=-\fint_{\om^\ep}\partial_3\psi^\ep\,\overline{\chi}\,dx.
\eeq
Moreover, using successively the Cauchy-Schwarz inequality, the boundedness of $\la^\ep_1(\eta)$ and the estimate \refe{aepbep}. satisfied by $\be_\ep$, we have
\beq\label{estD3psiep}
\left|\fint_{\om^\ep}\partial_3\psi^\ep\,\overline{\chi}\,dx\right|
\leq\left(\fint_{\om^\ep}|\nabla\psi^\ep|^2\,dx\right)^{1\over 2}\left(\fint_{\om^\ep}|\chi|^2\,dx\right)^{1\over 2}
\leq {c\over\sqrt{\be_\ep\,|\om^\ep|}}\,\left(\fint_{\om^\ep}|\chi|^2\,dx\right)^{1\over 2}=o(1).
\eeq
Then, passing to the limit in \refe{D3psiep}. thanks to \refe{tpsi*}. and \refe{estD3psiep}., we get that
\beq\label{d3tpsi*}
\int_{\bar{Y}}\partial_3\overline{\chi}\,d\tilde{\psi}^*=0.
\eeq
Writing $\chi=e^{i\,x\cdot\eta}\,\overline{\ph}$ with $\ph\in C^1_\sharp(Y;\CC)$, and using \refe{tph*}., equality \refe{d3tpsi*}. reads as
\beq\label{d3tph*}
\int_{\bar{Y}}\left(\partial_3\ph-i\,\eta_3\,\ph\right)d\tilde{\ph}^*=0.
\eeq
From now on assume that $\eta_3\notin\ZZ$.
Then, for $f\in C^1_\sharp(Y;\CC)$, we may define the function $\ph$ by
\beq\label{fph}
\ph(x',x_3):=e^{i\,\eta_3 x_3}\int_0^{x_3}e^{-i\,\eta_3 t}\,f(x',t)\,dt+{e^{i\,\eta_3 x_3}\over e^{-i\,2\pi\eta_3}-1}\int_0^{2\pi}e^{-i\,\eta_3 t}\,f(x',t)\,dt.
\eeq
It is easy to check that $\ph$ belongs to $C^1_\sharp(Y;\CC)$ and satisfies the equation $\partial_3\ph-i\,\eta_3\,\ph=f$ in $\RR^3$.
Therefore, from \refe{d3tph*}. we deduce that
\beq\label{limtphiep}
\int_{\bar{Y}}f\,d\tilde{\ph}^*=0,\quad\forall\,f\in C^1_\sharp(Y;\CC),
\eeq
or equivalently by \refe{tph*}.,
\beq\label{limtpsiep}
\int_{\bar{Y}}\overline{\chi}\,d\tilde{\psi}^*=0,\quad\forall\,\chi\in C^1_\eta(Y;\CC).
\eeq
\par
We can now determine the limit of the eigenvalue problem \refe{vppsiep}.. Let $\chi\in C^1_\eta(Y;\CC)$, putting the function $\chi\,\hat{v}^\ep$ defined by \refe{hvep}. as test function in \refe{vppsiep}. we have
\[
\int_Y \hat{v}_\ep\nabla\psi^\ep\cdot\nabla\overline{\chi}\,dx+\int_Y\nabla\hat{v}_\ep\cdot\nabla\psi^\ep\,\overline{\chi}\,dx
=\la^\ep_1(\eta)\int_Y\psi^\ep\,\overline{\chi}\,\hat{v}^\ep\,dx.
\]
Consider a subsequence of $\ep$, still denoted by $\ep$, such that $\la^\ep_1(\eta)$ converges to $\la^*_1(\eta)$. Then, passing to the limit in the previous equality thanks to the convergence of $\psi^\ep$ to $\psi^*$ in $H^1_\eta(Y;\CC)$, to \refe{conhvep}., and to the limit \refe{limhvepvep}. combined with equality \refe{limtpsiep}., we obtain the limit eigenvalue problem
\beq\label{limvppsi*}
\int_Y\nabla\psi^*\cdot\nabla\overline{\chi}\,dx+\ga\int_Y\psi^*\,\overline{\chi}\,dx=\la^*_1(\eta)\int_Y\psi^*\,\overline{\chi}\,dx,
\quad\forall\,\chi\in H^1_\eta(Y;\CC),
\eeq
where $\psi^*$ satisfies \refe{psi*}..
\par
It remains to prove that the limit of the first Bloch eigenvalue is given by
\beq\label{la*etaga}
\la_1^*(\eta)=\ga+\min\left\{\int_Y |\nabla\psi|^2\,dx\;:\;\psi\in H^1_\eta(Y;\CC)\mbox{ and }\int_Y|\psi|^2\,dx=1\right\},
\eeq
Let $\chi$ be a function in $C^1_\eta(Y;\CC)$ with a unit $L^2(Y)$-norm. Using \refe{laepeta}., \refe{conhvep}. and the convergence
\[
|\nabla\hat{v}^\ep|^2\;\rightharpoonup\;\lim_{\ep\to 0}\left({1\over\ep^2}\fint_Y|\nabla\hat{V}^\ep|^2\,dy\right)=\ga
\quad\mbox{weakly in }\M(\bar{Y})^2\,*,\quad\mbox{due to \refe{rep}. and \refe{hVep}.},
\]
we have
\[
\ba{ll}
\dis \la^\ep_1(\eta) & \dis \leq{1\over\|\chi\,\hat{v}^\ep\|^2_{L^2(Y)}}\int_Y a^\ep\big|\nabla(\chi\,\hat{v}^\ep)\big|^2\,dx
\\ \ecart
& \dis ={1\over\|\chi\,\hat{v}^\ep\|^2_{L^2(Y)}}\left(\int_Y |\nabla\hat{v}^\ep|^2\,|\chi|^2\,dx+\int_Y (\hat{v}^\ep)^2\,|\nabla\chi|^2\,dx
+2\int_Y\hat{v}^\ep\,\nabla\hat{v}^\ep\cdot\Re\left(\overline{\chi}\nabla\chi\right)dx\right)
\\ \ecart
& \dis =\ga+\int_Y |\nabla\chi|^2\,dx+o(1),
\ea
\]
which, by a density argument, implies that
\[
\la^*_1(\eta)\leq\ga+\int_Y |\nabla\psi|^2\,dx,\quad\forall\,\psi\in H^1_\eta(Y;\CC)\mbox{ with }\int_Y |\psi|^2\,dx=1.
\]
This combined with \refe{limvppsi*}. and \refe{psi*}. yields the minimization formula \refe{la*etaga}., which shows the uniqueness of the limit. Therefore, limit \refe{laepetaga}. holds for the whole sequence $\ep$. Finally, using the expansion in Fourier's series with $|\eta|\leq{1\over 2}$, formula \refe{laepetaga}. reduces to \refe{laepetaga0}.. The proof of Theorem~\ref{thm2} is thus complete. \cqfd
\par\bs\noindent
{\bf Acknowledgment.} The authors wish to thank J. Casado-D\'iaz for a stimulating discussion about Lemma~\ref{lemM} in connection with reference \cite{Man}. This work has been carried out within the project ``Homogenization and composites" supported by the {\em Indo French Center for Applied Mathematics - UMI IFCAM}. The authors are also grateful to the hospitality of TIFR-CAM Bangalore in February 2013 and INSA de Rennes in May 2013.
%
%
%
%

%
%
\end{document}